\documentclass[leqno,12pt]{amsart}
\usepackage{amssymb} 
\theoremstyle{plain} 
\newtheorem{theorem}{\indent\sc Theorem}[section] 

\newtheorem{proposition}[theorem]{\indent\sc Proposition}

\theoremstyle{definition} 

\begin{document}

\title[Coupled Painlev\'e III systems]{Coupled Painlev\'e III systems with affine Weyl group symmetry of types $B_5^{(1)},D_5^{(1)}$ and $D_6^{(2)}$ \\}
\author{Yusuke Sasano }

\renewcommand{\thefootnote}{\fnsymbol{footnote}}
\footnote[0]{2000\textit{ Mathematics Subjet Classification}.
Primary 34M55; Secondary 34M45.}

\maketitle

\begin{abstract}
We find and study four kinds of five-parameter family of six-dimensional coupled Painlev\'e III systems with affine Weyl group symmetry of types $D_5^{(1)},B_5^{(1)}$ and $D_6^{(2)}$. We show that each system is equivalent by an explicit birational and symplectic transformation, respectively. We also show that we characterize each system from the viewpoint of holomorphy.
\end{abstract}

\section{Introduction}

In \cite{Sasa1,Sasa2,Sasa5,Sasa3,Sasa4}, we presented some types of coupled Painlev\'e systems with various affine Weyl group symmetries. In this paper, we present a 5-parameter family of coupled Painlev\'e III systems with affine Weyl group symmetry of type $D_5^{(1)}$, which is explicitly given by
\begin{equation}\label{1}
\frac{dx}{dt}=\frac{\partial H}{\partial y}, \quad \frac{dy}{dt}=-\frac{\partial H}{\partial x}, \quad \frac{dz}{dt}=\frac{\partial H}{\partial w}, \quad \frac{dw}{dt}=-\frac{\partial H}{\partial z}, \quad \frac{dq}{dt}=\frac{\partial H}{\partial p}, \quad \frac{dp}{dt}=-\frac{\partial H}{\partial q}
\end{equation}
with the Hamiltonian
\begin{align}\label{2}
\begin{split}
H &=H_1(x,y,t;\alpha_0,\alpha_1)+H_{III}^{D_7^{(1)}}(z,w,t;\alpha_0+\alpha_1+2\alpha_2)\\
&+H_4(q,p,t;\alpha_4,\alpha_5)+\frac{2(xz-wp)}{t}\\
&=\frac{x^2y^2+xy^2-(\alpha_0+\alpha_1)xy-\alpha_0y}{t}+\frac{z^2w^2+(\alpha_0+\alpha_1+2\alpha_2)zw+z+tw}{t}\\
&+\frac{q^2p^2-tq^2p-(1-\alpha_4-\alpha_5)qp-\alpha_4tq}{t}+\frac{2(xz-wp)}{t}.
\end{split}
\end{align}
Here $x,y,z,w,q$ and $p$ denote unknown complex variables, and $\alpha_0,\alpha_1, \dots ,\alpha_5$ are complex parameters satisfying the relation:
\begin{equation}\label{3}
\alpha_0+\alpha_1+2\alpha_2+2\alpha_3+\alpha_4+\alpha_5=1.
\end{equation}
The symbols $H_1,H_2,H_3,H_4$ and $H_{III}^{D_7^{(1)}}$ denote the polynomial Hamiltonians explicitly given as follows:
\begin{align}
H_1&=H_1(q,p,t;\alpha_0,\alpha_1)=\frac{q^2p^2+qp^2-(\alpha_0+\alpha_1)qp-\alpha_0p}{t} \quad (\alpha_0-\alpha_1+2\alpha_2=0),
\end{align}
\begin{align}
H_2&=H_2(q,p,t;\alpha_2)=\frac{q^2p^2+(1-2\alpha_2)qp+tp}{t} \quad (\alpha_0-\alpha_1+2\alpha_2=0),\\
H_3&=H_3(q,p,t;\alpha_2)=\frac{q^2p^2+2\alpha_2qp-q}{t} \quad (\alpha_0-\alpha_1+2\alpha_2=0),\\
H_4&=H_4(q,p,t;\alpha_0,\alpha_1)=\frac{q^2p^2-tq^2p-(1-\alpha_0-\alpha_1)qp-\alpha_0tq}{t} \quad (\alpha_0-\alpha_1+2\alpha_2=0),\\
H_{III}^{D_7^{(1)}}&=H_{III}^{D_7^{(1)}}(q,p,t;\alpha_1)=\frac{q^2p^2+\alpha_1 qp+q+tp}{t} \quad (\alpha_0+\alpha_1=1).
\end{align}
We remark that for $y=q/{\tau}, \ t={\tau}^2$  the Hamiltonian system with $H_{III}^{D_7^{(1)}}$ is the special case of the third Painlev\'e system (see \cite{T}):
\begin{equation}
\frac{d^2y}{d{\tau}^2}=\frac{1}{y}\left(\frac{dy}{d{\tau}}\right)^2-\frac{1}{\tau}\frac{dy}{d{\tau}}+\frac{1}{\tau}(ay^2+b)+cy^3+\frac{d}{y}
\end{equation}
with
\begin{equation}
a=-8, \quad b=4(1-\alpha_1), \quad c=0, \quad d=-4.
\end{equation}

From the viewpoint of symmetry, the Hamiltonian system
\begin{equation}
\frac{dq}{dt}=\frac{\partial H_{III}^{D_7^{(1)}}}{\partial p}, \quad \frac{dp}{dt}=-\frac{\partial H_{III}^{D_7^{(1)}}}{\partial q}
\end{equation}
 has extended affine Weyl group symmetry of type $A_1^{(1)}$, whose generators $<s_0,s_1,\pi=\sigma \circ s_1>$ are explicitly given as follows (see \cite{T}):
\begin{equation}
  \left\{
  \begin{aligned}
   s_0(q,p,t;\alpha_0,\alpha_1) &=(q,p+\frac{\alpha_0}{q}-\frac{t}{q^2},-t;-\alpha_0,\alpha_1+2\alpha_0),\\
   s_1(q,p,t;\alpha_0,\alpha_1) &=(-q+\frac{\alpha_1}{p}+\frac{1}{p^2},-p,-t;\alpha_0+2\alpha_1,-\alpha_1),\\
   \sigma(q,p,t;\alpha_0,\alpha_1) &=(tp,-\frac{q}{t},-t;\alpha_1,\alpha_0).
   \end{aligned}
  \right. 
\end{equation}

\begin{proposition}\label{pro:1}
By the following birational and symplectic transformations $tr_i \ (i=1,2,3)${\rm : \rm}
\begin{equation}\label{0.1}
  \left\{
  \begin{aligned}
   tr_1(q,p) &=(t/p,(qp-\alpha_0)p/t),\\
   tr_2(q,p) &=(-tp,q/t),\\
   tr_3(q,p) &=(p/t,-tq),
   \end{aligned}
  \right. 
\end{equation}
the Hamiltonians $H_1,H_2,H_3$ and $H_4$ satisfy the following relations{\rm : \rm}
\begin{equation}
tr_1(H_1)=H_2, \quad tr_2 \circ tr_1(H_1)=H_3, \quad tr_3(H_1)=H_4.
\end{equation}
\end{proposition}
By Proposition \ref{pro:1}, we see that each Hamiltonian $H_i \ (i=1,2,3,4)$ is  equivalent by the transformations \eqref{0.1}.

The Hamiltonian system with $H_1$ has the first integral.
\begin{proposition}
The system with the Hamiltonian $H_1$
\begin{equation}
\frac{dq}{dt}=\frac{\partial H_1}{\partial p}, \quad \frac{dp}{dt}=-\frac{\partial H_1}{\partial q}
\end{equation}
has the first integral $I${\rm : \rm}
\begin{equation}
I=q^2p^2+qp^2-(\alpha_0+\alpha_1)qp-\alpha_0p.
\end{equation}
\end{proposition}
We see that the relation between the Hamiltonian $H_1$ and the first integral $I$ is explicitly given by
\begin{equation}
tH_1=I.
\end{equation}

The B{\"a}cklund transformations of this system satisfy Noumi-Yamada's universal description for $D_5^{(1)}$ root system (see \cite{NY1}). Since these universal B{\"a}cklund transformations have Lie theoretic origin, similarity reduction of a Drinfeld-Sokolov hierarchy admits such a B{\"a}cklund symmetry. The aim of this paper is to introduce the system of type $D_5^{(1)}$.

\begin{figure}[h]
\unitlength 0.1in
\begin{picture}(36.83,8.20)(0.27,-19.70)
\put(1.2000,-13.2000){\makebox(0,0)[lb]{Our discovery of the system}}%
\put(1.2000,-15.1000){\makebox(0,0)[lb]{of type $D_5^{(1)}$}}%
%
\put(0.2700,-14.2500){\makebox(0,0)[lb]{}}%
\put(24.1000,-14.4400){\makebox(0,0)[lb]{$W(D_5^{(1)})$}}%
%
\special{pn 20}%
\special{sh 0.600}%
\special{ar 403 1675 12 20  0.0000000 6.2831853}%
\put(4.7300,-17.6200){\makebox(0,0)[lb]{Noumi-Yamada's universal description}}%
\put(37.0000,-17.2000){\makebox(0,0)[lb]{Drinfeld-Sokolov}}%
\put(4.6700,-19.6100){\makebox(0,0)[lb]{for $D_5^{(1)}$ root system}}%
%
\special{pn 8}%
\special{pa 280 1583}%
\special{pa 280 1970}%
\special{fp}%
\special{pa 280 1550}%
\special{pa 3360 1550}%
\special{fp}%
%
\special{pn 8}%
\special{pa 280 1970}%
\special{pa 3351 1970}%
\special{fp}%
%
\special{pn 8}%
\special{pa 3351 1550}%
\special{pa 3351 1959}%
\special{fp}%
%
\special{pn 20}%
\special{pa 2120 1270}%
\special{pa 2148 1246}%
\special{pa 2176 1225}%
\special{pa 2204 1209}%
\special{pa 2232 1201}%
\special{pa 2261 1202}%
\special{pa 2290 1213}%
\special{pa 2317 1234}%
\special{pa 2340 1260}%
\special{pa 2359 1291}%
\special{pa 2371 1324}%
\special{pa 2377 1358}%
\special{pa 2376 1391}%
\special{pa 2366 1421}%
\special{pa 2346 1446}%
\special{pa 2320 1465}%
\special{pa 2289 1478}%
\special{pa 2257 1483}%
\special{pa 2226 1483}%
\special{pa 2194 1479}%
\special{pa 2162 1473}%
\special{pa 2150 1470}%
\special{sp}%
%
\special{pn 20}%
\special{pa 2190 1480}%
\special{pa 2070 1420}%
\special{fp}%
\special{sh 1}%
\special{pa 2070 1420}%
\special{pa 2121 1468}%
\special{pa 2118 1444}%
\special{pa 2139 1432}%
\special{pa 2070 1420}%
\special{fp}%
%
\special{pn 20}%
\special{pa 3440 1751}%
\special{pa 3680 1751}%
\special{fp}%
\special{sh 1}%
\special{pa 3680 1751}%
\special{pa 3613 1731}%
\special{pa 3627 1751}%
\special{pa 3613 1771}%
\special{pa 3680 1751}%
\special{fp}%
\put(37.1000,-19.0000){\makebox(0,0)[lb]{hierarchy?}}%
\end{picture}%
\label{fig:D51}
\caption{}
\end{figure}

\noindent
We also show that this system coincides with some types of 5-parameter family of six-dimensional coupled Painlev\'e III systems with extended affine Weyl group symmetry of types $B_5^{(1)}$ and $D_6^{(2)}$. Moreover, we show the relationship between the system of type $D_5^{(1)}$ and the system of types $B_5^{(1)}$ and $D_6^{(2)}$ by an explicit birational and symplectic transformation, respectively.

This paper is organized as follows. In Section 2, we will introduce the system of type $D_5^{(1)}$ and its B{\"a}cklund transformations. In Section 3, we will propose two types of a 5-parameter family of coupled Painlev\'e III systems in dimension six with extended affine Weyl group symmetry of type $B_5^{(1)}$. We also show that each of them is equivalent to the system \eqref{1} by a birational and symplectic transformation, respectively. In Section 4, we will propose a 5-parameter family of coupled Painlev\'e III systems in dimension six with extended affine Weyl group symmetry of type $D_6^{(2)}$ and its B{\"a}cklund transformations. We also show that this system is equivalent to the system \eqref{1} by a birational and symplectic transformation.

\section{The system of type $D_5^{(1)}$}

In this section, we present a 5-parameter family of polynomial Hamiltonian systems that can be considered as six-dimensional coupled Painlev\'e III systems given by
\begin{equation}\label{4}
  \left\{
  \begin{aligned}
   \frac{dx}{dt} &=\frac{2x^2y+2xy-(\alpha_0+\alpha_1)x-\alpha_0}{t},\\
   \frac{dy}{dt} &=-\frac{2xy^2+y^2-(\alpha_0+\alpha_1)y+2z}{t},\\
   \frac{dz}{dt} &=\frac{2z^2w+(\alpha_0+\alpha_1+2\alpha_2)z+t-2p}{t},\\
   \frac{dw}{dt} &=-\frac{2zw^2+(\alpha_0+\alpha_1+2\alpha_2)w+1+2x}{t},\\
   \frac{dq}{dt} &=\frac{2q^2p-tq^2-(1-\alpha_4-\alpha_5)q-2w}{t},\\
   \frac{dp}{dt} &=\frac{-2qp^2+2tqp+(1-\alpha_4-\alpha_5)p+\alpha_4t}{t}
   \end{aligned}
  \right. 
\end{equation}
with the Hamiltonian \eqref{2}.

\begin{theorem}\label{2.1}
The system \eqref{4} admits extended affine Weyl group symmetry of type $D_5^{(1)}$ as the group of its B{\"a}cklund transformations {\rm(cf. \cite{NY2,S,T}),\rm} whose generators are explicitly given as follows{\rm : \rm}with the notation $(*):=(x,y,z,w,q,p,t;\alpha_0,\alpha_1, \dots ,\alpha_5),$
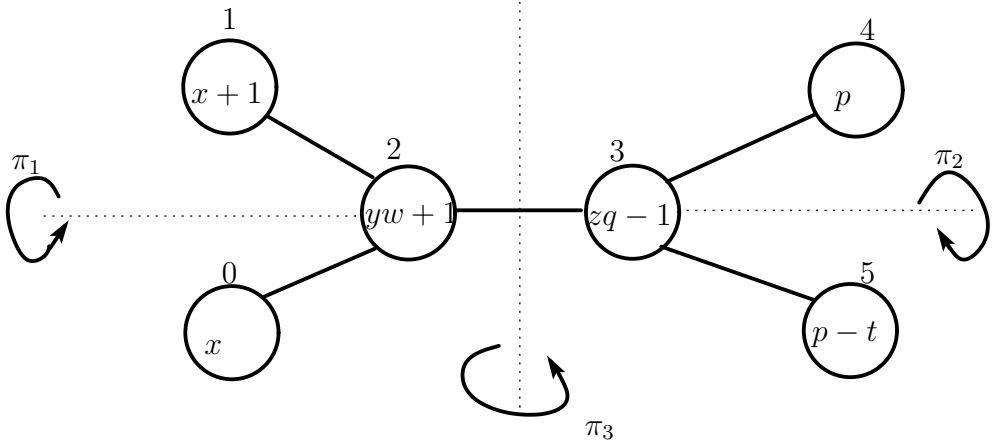
\begin{figure}[h]
\unitlength 0.1in
\begin{picture}(51.30,21.92)(13.00,-25.52)
%
\special{pn 20}%
\special{ar 2462 834 244 244  0.0000000 6.2831853}%
%
\special{pn 20}%
\special{ar 2473 2121 244 244  0.0000000 6.2831853}%
%
\special{pn 20}%
\special{ar 3397 1494 244 244  0.0000000 6.2831853}%
%
\special{pn 20}%
\special{pa 2660 988}%
\special{pa 3210 1307}%
\special{fp}%
%
\special{pn 20}%
\special{pa 2638 1934}%
\special{pa 3221 1681}%
\special{fp}%
\put(23.3000,-22.3000){\makebox(0,0)[lb]{$x$}}%
\put(22.6000,-9.3000){\makebox(0,0)[lb]{$x+1$}}%
\put(31.7000,-15.8000){\makebox(0,0)[lb]{$yw+1$}}%
\put(24.2000,-18.6000){\makebox(0,0)[lb]{$0$}}%
\put(24.2000,-5.3000){\makebox(0,0)[lb]{$1$}}%
\put(32.8000,-12.1000){\makebox(0,0)[lb]{$2$}}%
%
\special{pn 20}%
\special{pa 3640 1480}%
\special{pa 4300 1480}%
\special{fp}%
%
\special{pn 20}%
\special{ar 4570 1490 244 244  0.0000000 6.2831853}%
%
\special{pn 20}%
\special{pa 4750 1330}%
\special{pa 5520 980}%
\special{fp}%
%
\special{pn 20}%
\special{pa 4720 1670}%
\special{pa 5520 1950}%
\special{fp}%
%
\special{pn 20}%
\special{ar 5740 850 244 244  0.0000000 6.2831853}%
%
\special{pn 20}%
\special{ar 5710 2110 244 244  0.0000000 6.2831853}%
\put(43.3000,-15.9000){\makebox(0,0)[lb]{$zq-1$}}%
\put(55.1000,-22.0000){\makebox(0,0)[lb]{$p-t$}}%
\put(56.3000,-9.6000){\makebox(0,0)[lb]{$p$}}%
\put(44.5000,-12.2000){\makebox(0,0)[lb]{$3$}}%
\put(57.6000,-18.6000){\makebox(0,0)[lb]{$5$}}%
\put(57.6000,-5.9000){\makebox(0,0)[lb]{$4$}}%
%
\special{pn 8}%
\special{pa 1490 1510}%
\special{pa 3120 1510}%
\special{dt 0.045}%
\special{pa 3120 1510}%
\special{pa 3119 1510}%
\special{dt 0.045}%
%
\special{pn 8}%
\special{pa 4810 1480}%
\special{pa 6350 1480}%
\special{dt 0.045}%
\special{pa 6350 1480}%
\special{pa 6349 1480}%
\special{dt 0.045}%
%
\special{pn 20}%
\special{pa 1566 1408}%
\special{pa 1550 1377}%
\special{pa 1532 1349}%
\special{pa 1510 1327}%
\special{pa 1484 1313}%
\special{pa 1452 1310}%
\special{pa 1418 1315}%
\special{pa 1386 1327}%
\special{pa 1357 1344}%
\special{pa 1335 1366}%
\special{pa 1319 1392}%
\special{pa 1307 1422}%
\special{pa 1301 1454}%
\special{pa 1299 1488}%
\special{pa 1302 1523}%
\special{pa 1308 1560}%
\special{pa 1318 1596}%
\special{pa 1331 1633}%
\special{pa 1346 1668}%
\special{pa 1364 1701}%
\special{pa 1385 1728}%
\special{pa 1410 1745}%
\special{pa 1441 1747}%
\special{pa 1474 1734}%
\special{pa 1502 1711}%
\special{pa 1516 1681}%
\special{pa 1516 1668}%
\special{sp}%
%
\special{pn 20}%
\special{pa 6070 1440}%
\special{pa 6090 1406}%
\special{pa 6111 1374}%
\special{pa 6132 1344}%
\special{pa 6153 1319}%
\special{pa 6176 1299}%
\special{pa 6199 1285}%
\special{pa 6224 1280}%
\special{pa 6250 1283}%
\special{pa 6276 1295}%
\special{pa 6303 1313}%
\special{pa 6328 1336}%
\special{pa 6353 1364}%
\special{pa 6375 1395}%
\special{pa 6394 1427}%
\special{pa 6409 1461}%
\special{pa 6421 1494}%
\special{pa 6428 1527}%
\special{pa 6430 1560}%
\special{pa 6428 1592}%
\special{pa 6420 1623}%
\special{pa 6407 1654}%
\special{pa 6388 1683}%
\special{pa 6363 1710}%
\special{pa 6335 1731}%
\special{pa 6307 1740}%
\special{pa 6281 1734}%
\special{pa 6258 1714}%
\special{pa 6237 1686}%
\special{pa 6220 1660}%
\special{sp}%
%
\special{pn 20}%
\special{pa 6230 1680}%
\special{pa 6180 1600}%
\special{fp}%
\special{sh 1}%
\special{pa 6180 1600}%
\special{pa 6198 1667}%
\special{pa 6208 1645}%
\special{pa 6232 1646}%
\special{pa 6180 1600}%
\special{fp}%
\put(13.2000,-12.9000){\makebox(0,0)[lb]{$\pi_1$}}%
\put(61.5000,-12.6000){\makebox(0,0)[lb]{$\pi_2$}}%
%
\special{pn 20}%
\special{pa 1520 1680}%
\special{pa 1600 1560}%
\special{fp}%
\special{sh 1}%
\special{pa 1600 1560}%
\special{pa 1546 1604}%
\special{pa 1570 1604}%
\special{pa 1580 1627}%
\special{pa 1600 1560}%
\special{fp}%
%
\special{pn 8}%
\special{pa 3980 390}%
\special{pa 3980 2510}%
\special{dt 0.045}%
\special{pa 3980 2510}%
\special{pa 3980 2509}%
\special{dt 0.045}%
\put(43.2000,-26.8000){\makebox(0,0)[lb]{$\pi_3$}}%
%
\special{pn 20}%
\special{pa 3870 2190}%
\special{pa 3835 2200}%
\special{pa 3800 2210}%
\special{pa 3768 2223}%
\special{pa 3739 2238}%
\special{pa 3715 2256}%
\special{pa 3696 2279}%
\special{pa 3684 2306}%
\special{pa 3679 2338}%
\special{pa 3682 2372}%
\special{pa 3692 2404}%
\special{pa 3710 2433}%
\special{pa 3734 2460}%
\special{pa 3765 2483}%
\special{pa 3800 2502}%
\special{pa 3839 2519}%
\special{pa 3880 2532}%
\special{pa 3924 2542}%
\special{pa 3968 2548}%
\special{pa 4012 2551}%
\special{pa 4054 2551}%
\special{pa 4094 2548}%
\special{pa 4131 2541}%
\special{pa 4164 2531}%
\special{pa 4191 2517}%
\special{pa 4212 2500}%
\special{pa 4225 2480}%
\special{pa 4230 2456}%
\special{pa 4227 2429}%
\special{pa 4216 2400}%
\special{pa 4201 2369}%
\special{pa 4184 2337}%
\special{pa 4180 2330}%
\special{sp}%
%
\special{pn 20}%
\special{pa 4170 2320}%
\special{pa 4140 2270}%
\special{fp}%
\special{sh 1}%
\special{pa 4140 2270}%
\special{pa 4157 2337}%
\special{pa 4167 2316}%
\special{pa 4191 2317}%
\special{pa 4140 2270}%
\special{fp}%
\end{picture}%
\label{fig:D52}
\caption{Dynkin diagram of type $D_5^{(1)}$}
\end{figure}
\begin{align*}
        s_{0}: (*) &\rightarrow (x,y-\frac{\alpha_0}{x},z,w,q,p,t;-\alpha_0,\alpha_1,\alpha_2+\alpha_0,\alpha_3,\alpha_4,\alpha_5), \\
        s_{1}: (*) &\rightarrow (x,y-\frac{\alpha_1}{x+1},z,w,q,p,t;\alpha_0,-\alpha_1,\alpha_2+\alpha_1,\alpha_3,\alpha_4,\alpha_5), 
\end{align*}
\begin{align*}
        s_{2}: (*) &\rightarrow  (x+\frac{\alpha_2w}{yw+1},y,z+\frac{\alpha_2y}{yw+1},w,q,p,t;\alpha_0+\alpha_2,\alpha_1+\alpha_2,-\alpha_2,\alpha_3+\alpha_2,\alpha_4,\alpha_5), \\
        s_{3}: (*) &\rightarrow (x,y,z,w-\frac{\alpha_3q}{zq-1},q,p-\frac{\alpha_3z}{zq-1},t;\alpha_0,\alpha_1,\alpha_2+\alpha_3,-\alpha_3,\alpha_4+\alpha_3,\alpha_5+\alpha_3), \\
        s_{4}: (*) &\rightarrow (x,y,z,w,q+\frac{\alpha_4}{p},p,t;\alpha_0,\alpha_1,\alpha_2,\alpha_3+\alpha_4,-\alpha_4,\alpha_5), \\
        s_{5}: (*) &\rightarrow (x,y,z,w,q+\frac{\alpha_5}{p-t},p,t;\alpha_0,\alpha_1,\alpha_2,\alpha_3+\alpha_5,\alpha_4,-\alpha_5), \\
        \pi_1: (*) &\rightarrow (-x-1,-y,-z,-w,-q,-p,-t;\alpha_1,\alpha_0,\alpha_2,\alpha_3,\alpha_4,\alpha_5), \\
        \pi_2: (*) &\rightarrow (x,y,z,w,q,p-t,-t;\alpha_0,\alpha_1,\alpha_2,\alpha_3,\alpha_5,\alpha_4), \\
        \pi_3: (*) &\rightarrow (\frac{(p-t)}{t},-tq,-tw,\frac{z}{t},\frac{y}{t},-t(x+1),-t;\alpha_5,\alpha_4,\alpha_3,\alpha_2,\alpha_1,\alpha_0).
\end{align*}
\end{theorem}

\begin{theorem}\label{2.2}
Let us consider a polynomial Hamiltonian system with Hamiltonian $H \in {\Bbb C}(t)[x,y,z,w,q,p]$. We assume that

$(A1) \ deg(H)=4$ with respect to $x,y,z,w,q,p$.

$(A2)$ This system becomes again a polynomial Hamiltonian system in each coordinate $r_i \ (i=0,1,4,5)${\rm : \rm}
\begin{align*}
r_0&:x_0=-(xy-\alpha_0)y, \ y_0=1/y, \ z_0=z, \ w_0=w, \ q_0=q, \ p_0=p, \\
r_1&:x_1=-((x+1)y-\alpha_1)y, \ y_1=1/y, \ z_1=z, \ w_1=w, \ q_1=q, \ p_1=p, \\
r_4&:x_4=x, \ y_4=y, \ z_4=z, \ w_4=w, \ q_4=1/q, \ p_4=-(pq+\alpha_4)q, \\
r_5&:x_5=x, \ y_5=y, \ z_5=z, \ w_5=w. \ q_5=1/q, \ p_5=-((p-t)q+\alpha_5)q.
\end{align*}

$(A3)$ In addition to the assumption $(A2)$, the Hamiltonian system in each coordinate system $(x_i,y_i,z_i,w_i,q_i,p_i) \ (i=0,4)$ becomes again a polynomial Hamiltonian system in each coordinate $r_i \ (i=2,3)$, respectively{\rm : \rm}
\begin{align*}
r_2&:x_2=1/x_0, \ y_2=-((y_0+w_0)x_0+\alpha_2)x_0, \ z_2=z_0-x_0, \ w_2=w_0, \ q_2=q_0, \ p_2=p_0, \\
r_3&:x_3=x_4, \ y_3=y_4, \ z_3=-((z_4-q_4)w_4-\alpha_3)w_4, \ w_3=1/w_4, \ q_3=q_4, \ p_3=p_4+w_4.
\end{align*}
Then such a system coincides with the system \eqref{4}.
\end{theorem}

\section{The system of type $B_5^{(1)}$ }
In this section, we propose two types of a 5-parameter family of coupled Painlev\'e III systems in dimension six with extended affine Weyl group symmetry of type $B_5^{(1)}$. Each of them is equivalent to a polynomial Hamiltonian system, however, each has a different representaion of type $B_5^{(1)}$. We also show that each of them is equivalent to the system \eqref{4} by a birational and symplectic transformation, respectively.

The first member is given by
\begin{equation}\label{5}
  \left\{
  \begin{aligned}
   \frac{dx}{dt} &=\frac{2x^2y+2\alpha_0x}{t}-\frac{4xyz-2\alpha_1z}{t},\\
   \frac{dy}{dt} &=-\frac{2xy^2+2\alpha_0y-1}{t}+\frac{2y^2z}{t},\\
   \frac{dz}{dt} &=\frac{2z^2w+2(\alpha_0+\alpha_1+\alpha_2)z+t}{t}-\frac{2p}{t},\\
   \frac{dw}{dt} &=-\frac{2zw^2+2(\alpha_0+\alpha_1+\alpha_2)w+1}{t}+\frac{2(xy-\alpha_1)y}{t},\\
   \frac{dq}{dt} &=\frac{2q^2p-tq^2-2(\alpha_0+\alpha_1+\alpha_2+\alpha_3)q}{t}-\frac{2w}{t},\\
   \frac{dp}{dt} &=\frac{-2qp^2+2tqp+2(\alpha_0+\alpha_1+\alpha_2+\alpha_3)p+\alpha_4t}{t}
   \end{aligned}
  \right. 
\end{equation}
with the Hamiltonian
\begin{align}\label{6}
\begin{split}
H &=H_3(x,y,t;\alpha_0)+H_{III}^{D_7^{(1)}}(z,w,t,2(\alpha_0+\alpha_1+\alpha_2))\\
&+H_4(q,p,t;\alpha_4,\alpha_5)-\frac{2(xy-\alpha_1)yz+2wp}{t}\\
&=\frac{x^2y^2+2\alpha_0xy-x}{t}+\frac{z^2w^2+2(\alpha_0+\alpha_1+\alpha_2)zw+z+tw}{t}\\
&+\frac{q^2p^2-tq^2p-2(\alpha_0+\alpha_1+\alpha_2+\alpha_3)qp-\alpha_4tq}{t}-\frac{2(xy-\alpha_1)yz+2wp}{t}.
\end{split}
\end{align}
Here $x,y,z,w,q$ and $p$ denote unknown complex variables and $\alpha_0,\alpha_1, \dots ,\alpha_5$ are complex parameters satisfying the relation:
\begin{equation}\label{7}
2\alpha_0+2\alpha_1+2\alpha_2+2\alpha_3+\alpha_4+\alpha_5=1.
\end{equation}

\begin{theorem}\label{3.1}
The system \eqref{5} admits extended affine Weyl group symmetry of type $B_5^{(1)}$ as the group of its B{\"a}cklund transformations {\rm (cf. \cite{NY2,S,T}), \rm} whose generators are explicitly given as follows{\rm : \rm}
\begin{align*}
        s_{0}: (*) &\rightarrow (-x-\frac{2\alpha_0}{y}+\frac{1}{y^2},-y,-z,-w,-q,-p,-t;-\alpha_0,\alpha_1+2\alpha_0,\alpha_2,\alpha_3,\alpha_4,\alpha_5), \\
        s_{1}: (*) &\rightarrow (x,y-\frac{\alpha_1}{x},z,w,q,p,t;\alpha_0+\alpha_1,-\alpha_1,\alpha_2+\alpha_1,\alpha_3,\alpha_4,\alpha_5), \\
        s_{2}: (*) &\rightarrow  (x+\frac{\alpha_2}{y+w},y,z+\frac{\alpha_2}{y+w},w,q,p,t;\alpha_0,\alpha_1+\alpha_2,-\alpha_2,\alpha_3+\alpha_2,\alpha_4,\alpha_5),\\
        s_{3}: (*) &\rightarrow (x,y,z,w-\frac{\alpha_3q}{zq-1},q,p-\frac{\alpha_3z}{zq-1},t;\alpha_0,\alpha_1,\alpha_2+\alpha_3,-\alpha_3,\alpha_4+\alpha_3,\alpha_5+\alpha_3), \\
        s_{4}: (*) &\rightarrow (x,y,z,w,q+\frac{\alpha_4}{p},p,t;\alpha_0,\alpha_1,\alpha_2,\alpha_3+\alpha_4,-\alpha_4,\alpha_5), \\
        s_{5}: (*) &\rightarrow (x,y,z,w,q+\frac{\alpha_5}{p-t},p,t;\alpha_0,\alpha_1,\alpha_2,\alpha_3+\alpha_5,\alpha_4,-\alpha_5), \\
        \varphi: (*) &\rightarrow (x,y,z,w,q,p-t,-t;\alpha_0,\alpha_1,\alpha_2,\alpha_3,\alpha_5,\alpha_4).
\end{align*}
\end{theorem}

\begin{figure}[t]
\unitlength 0.1in
\begin{picture}(44.42,18.96)(16.38,-23.66)
%
\special{pn 20}%
\special{ar 2602 1562 232 232  0.0000000 6.2831853}%
%
\special{pn 20}%
\special{ar 1870 1560 232 232  0.0000000 6.2831853}%
%
\special{pn 20}%
\special{ar 4442 1558 232 232  0.0000000 6.2831853}%
%
\special{pn 20}%
\special{ar 5332 912 232 232  0.0000000 6.2831853}%
%
\special{pn 20}%
\special{ar 5343 2134 232 232  0.0000000 6.2831853}%
%
\special{pn 20}%
\special{pa 4609 1391}%
\special{pa 5194 1108}%
\special{fp}%
%
\special{pn 20}%
\special{pa 4588 1735}%
\special{pa 5121 1986}%
\special{fp}%
\put(24.5600,-16.5600){\makebox(0,0)[lb]{$x$}}%
\put(42.0000,-16.5000){\makebox(0,0)[lb]{$zq-1$}}%
\put(52.1000,-10.1000){\makebox(0,0)[lb]{$p$}}%
\put(51.4000,-22.3000){\makebox(0,0)[lb]{$p-t$}}%
\put(17.5000,-12.9000){\makebox(0,0)[lb]{$0$}}%
\put(24.8000,-12.8000){\makebox(0,0)[lb]{$1$}}%
\put(43.2200,-12.9000){\makebox(0,0)[lb]{$3$}}%
\put(52.2200,-18.9000){\makebox(0,0)[lb]{$5$}}%
\put(52.2000,-6.4000){\makebox(0,0)[lb]{$4$}}%
%
\special{pn 20}%
\special{pa 2840 1570}%
\special{pa 3270 1570}%
\special{fp}%
%
\special{pn 20}%
\special{pa 2370 1510}%
\special{pa 2100 1510}%
\special{fp}%
\special{sh 1}%
\special{pa 2100 1510}%
\special{pa 2167 1530}%
\special{pa 2153 1510}%
\special{pa 2167 1490}%
\special{pa 2100 1510}%
\special{fp}%
%
\special{pn 20}%
\special{pa 2370 1630}%
\special{pa 2100 1630}%
\special{fp}%
\special{sh 1}%
\special{pa 2100 1630}%
\special{pa 2167 1650}%
\special{pa 2153 1630}%
\special{pa 2167 1610}%
\special{pa 2100 1630}%
\special{fp}%
%
\special{pn 8}%
\special{pa 4682 1580}%
\special{pa 5912 1580}%
\special{dt 0.045}%
\special{pa 5912 1580}%
\special{pa 5911 1580}%
\special{dt 0.045}%
%
\special{pn 20}%
\special{pa 5722 1370}%
\special{pa 5742 1336}%
\special{pa 5763 1304}%
\special{pa 5784 1274}%
\special{pa 5805 1250}%
\special{pa 5828 1232}%
\special{pa 5852 1222}%
\special{pa 5877 1221}%
\special{pa 5903 1229}%
\special{pa 5930 1245}%
\special{pa 5956 1268}%
\special{pa 5981 1294}%
\special{pa 6004 1324}%
\special{pa 6025 1355}%
\special{pa 6042 1385}%
\special{pa 6056 1416}%
\special{pa 6067 1447}%
\special{pa 6075 1478}%
\special{pa 6079 1510}%
\special{pa 6080 1542}%
\special{pa 6077 1574}%
\special{pa 6071 1607}%
\special{pa 6062 1640}%
\special{pa 6049 1674}%
\special{pa 6033 1707}%
\special{pa 6015 1738}%
\special{pa 5994 1765}%
\special{pa 5971 1787}%
\special{pa 5945 1803}%
\special{pa 5919 1810}%
\special{pa 5891 1809}%
\special{pa 5862 1801}%
\special{pa 5833 1786}%
\special{pa 5803 1768}%
\special{pa 5772 1747}%
\special{pa 5762 1740}%
\special{sp}%
%
\special{pn 20}%
\special{pa 5782 1750}%
\special{pa 5672 1650}%
\special{fp}%
\special{sh 1}%
\special{pa 5672 1650}%
\special{pa 5708 1710}%
\special{pa 5711 1686}%
\special{pa 5735 1680}%
\special{pa 5672 1650}%
\special{fp}%
\put(57.5200,-12.0000){\makebox(0,0)[lb]{$\varphi$}}%
%
\special{pn 20}%
\special{ar 3522 1572 232 232  0.0000000 6.2831853}%
\put(33.2000,-16.7000){\makebox(0,0)[lb]{$y+w$}}%
\put(34.0000,-12.9000){\makebox(0,0)[lb]{$2$}}%
%
\special{pn 20}%
\special{pa 3760 1580}%
\special{pa 4190 1580}%
\special{fp}%
\end{picture}%
\label{fig:D53}
\caption{Dynkin diagram of type $B_5^{(1)}$}
\end{figure}
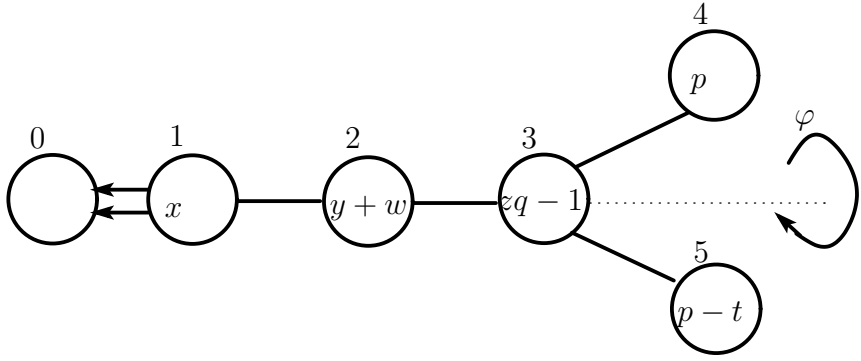

\begin{theorem}\label{3.2}
Let us consider a polynomial Hamiltonian system with Hamiltonian $H \in {\Bbb C}(t)[x,y,z,w,q,p]$. We assume that

$(A1)$ $deg(H)=4$ with respect to $x,y,z,w,q,p$.

$(A2)$ This system becomes again a polynomial Hamiltonian system in each coordinate $r_i \ (i=0,1,2,4,5)${\rm : \rm}
\begin{align*}
r_0&:x_0=x+\frac{2\alpha_0}{y}-\frac{1}{y^2}, \ y_0=y, \ z_0=z, \ w_0=w, \ q_0=q, \ p_0=p, \\
r_1&:x_1=-(xy-\alpha_1)y, \ y_1=1/y, \ z_1=z, \ w_1=w, \ q_1=q, \ p_1=p, \\
r_2&:x_2=1/x, \ y_2=-((y+w)x+\alpha_2)x, \ z_2=z-x, \ w_2=w, \ q_2=q, \ p_2=p, \\
r_4&:x_4=x, \ y_4=y, \ z_4=z, \ w_4=w, \ q_4=1/q, \ p_4=-(pq+\alpha_4)q, \\
r_5&:x_5=x, \ y_5=y, \ z_5=z, \ w_5=w. \ q_5=1/q, \ p_5=-((p-t)q+\alpha_5)q.
\end{align*}

$(A3)$ In addition to the assumption $(A2)$, the Hamiltonian system in the coordinate system $(x_4,y_4,z_4,w_4,q_4,p_4)$ becomes again a polynomial Hamiltonian system in the coordinate $r_3$\rm{:\rm}
\begin{align*}
r_3:x_3=x_4, \ y_3=y_4, \ z_3=-((z_4-q_4)w_4-\alpha_3)w_4, \ w_3=1/w_4, \ q_3=q_4, \ p_3=p_4+w_4.
\end{align*}
Then such a system coincides with the system \eqref{5}.
\end{theorem}

Theorems \ref{3.1} and \ref{3.2} can be checked by a direct calculation, respectively.

\begin{theorem}\label{3.3}
For the system \eqref{4} of type $D_5^{(1)}$, we make the change of parameters and variables
\begin{gather}
\begin{gathered}\label{8}
A_0=\frac{\alpha_1-\alpha_0}{2}, \quad A_1=\alpha_0, \quad A_2=\alpha_2, \quad A_3=\alpha_3, \quad A_4=\alpha_4, \quad A_5=\alpha_5,\\
\end{gathered}\\
\begin{gathered}\label{9}
X=-(xy-\alpha_0)y, \quad Y=\frac{1}{y}, \quad Z=z, \quad W=w, \quad Q=q, \quad P=p
\end{gathered}
\end{gather}
from $\alpha_0,\alpha_1, \dots ,\alpha_5,x,y,z,w,q,p$ to $A_0,A_1,\dots ,A_5,X,Y,Z,W,Q,P$. Then the system \eqref{4} can also be written in the new variables $X,Y,Z,W,Q,P$ and parameters $A_0,A_1,\dots ,A_5$ as a Hamiltonian system. This new system tends to the system \eqref{5} with the Hamiltonian \eqref{6}.
\end{theorem}

\begin{proof}
Notice that
$$
2A_0+2A_1+2A_2+2A_3+A_4+A_5=\alpha_0+\alpha_1+2\alpha_2+2\alpha_3+\alpha_4+\alpha_5=1
$$
and the change of variables from $(x,y,z,w,q,p)$ to $(X,Y,Z,W,Q,P)$ is symplectic. Choose $S_i$ $(i=0,1,\dots ,5)$ and $\varphi$ as
$$
S_0:=\pi_1, \ S_1:=s_0, \ S_2:=s_2, \ S_3:=s_3, \ S_4:=s_4, \ S_5:=s_5, \ \varphi:=\pi_2.
$$
Then the transformations $S_i$ are reflections of the parameters $A_0,A_1,\dots ,A_5$. The transformation group $\tilde{W}(B_5^{(1)})=<S_0,S_1,\dots ,S_5,\varphi>$ coincides with the transformations given in Theorem \ref{3.1}.
\end{proof}

The second member is given by
\begin{equation}\label{10}
  \left\{
  \begin{aligned}
   \frac{dx}{dt} &=\frac{2x^2y+2xy-(\alpha_0+\alpha_1)x-\alpha_0}{t},\\
   \frac{dy}{dt} &=-\frac{2xy^2+y^2-(\alpha_0+\alpha_1)y}{t}-\frac{2z}{t},\\
   \frac{dz}{dt} &=\frac{2z^2w+(\alpha_0+\alpha_1+2\alpha_2)z+t}{t}+\frac{2q(qp+\alpha_4)}{t},\\
   \frac{dw}{dt} &=-\frac{2zw^2+(\alpha_0+\alpha_1+2\alpha_2)w+1}{t}-\frac{2x}{t},\\
   \frac{dq}{dt} &=\frac{2q^2p-(2\alpha_5-1)q+t}{t}+\frac{2wq^2}{t},\\
   \frac{dp}{dt} &=\frac{-2qp^2+(2\alpha_5-1)p}{t}-\frac{4wqp+2\alpha_4w}{t}
   \end{aligned}
  \right. 
\end{equation}
with the Hamiltonian
\begin{align}\label{11}
\begin{split}
H &=H_1(x,y,t;\alpha_0,\alpha_1)+H_{III}^{D_7^{(1)}}(z,w,t,\alpha_0+\alpha_1+2\alpha_2)\\
&+H_2(q,p,t;\alpha_5)+\frac{2xz+2wq(qp+\alpha_4)}{t}\\
&=\frac{x^2y^2+xy^2-(\alpha_0+\alpha_1)xy-\alpha_0y}{t}+\frac{z^2w^2+(\alpha_0+\alpha_1+2\alpha_2)zw+z+tw}{t}\\
&+\frac{q^2p^2-(2\alpha_5-1)qp+tp}{t}+\frac{2xz+2wq(qp+\alpha_4)}{t}.
\end{split}
\end{align}
Here $x,y,z,w,q$ and $p$ denote unknown complex variables and $\alpha_0,\alpha_1, \dots ,\alpha_5$ are complex parameters satisfying the relation:
\begin{equation}\label{12}
\alpha_0+\alpha_1+2\alpha_2+2\alpha_3+2\alpha_4+2\alpha_5=1.
\end{equation}

\begin{figure}
\unitlength 0.1in
\begin{picture}(41.42,18.76)(15.30,-24.16)
%
\special{pn 20}%
\special{ar 2212 962 232 232  0.0000000 6.2831853}%
%
\special{pn 20}%
\special{ar 2223 2184 232 232  0.0000000 6.2831853}%
%
\special{pn 20}%
\special{ar 3101 1589 232 232  0.0000000 6.2831853}%
%
\special{pn 20}%
\special{pa 2400 1108}%
\special{pa 2923 1411}%
\special{fp}%
%
\special{pn 20}%
\special{pa 2380 2007}%
\special{pa 2933 1766}%
\special{fp}%
\put(20.7000,-22.8000){\makebox(0,0)[lb]{$x$}}%
\put(20.0000,-10.5000){\makebox(0,0)[lb]{$x+1$}}%
\put(28.6000,-16.9000){\makebox(0,0)[lb]{$yw+1$}}%
%
\special{pn 20}%
\special{pa 4120 1600}%
\special{pa 4400 1600}%
\special{fp}%
%
\special{pn 20}%
\special{ar 4650 1610 232 232  0.0000000 6.2831853}%
%
\special{pn 20}%
\special{ar 5440 1600 232 232  0.0000000 6.2831853}%
%
\special{pn 20}%
\special{pa 4880 1550}%
\special{pa 5240 1550}%
\special{fp}%
\special{sh 1}%
\special{pa 5240 1550}%
\special{pa 5173 1530}%
\special{pa 5187 1550}%
\special{pa 5173 1570}%
\special{pa 5240 1550}%
\special{fp}%
%
\special{pn 20}%
\special{pa 4870 1710}%
\special{pa 5220 1710}%
\special{fp}%
\special{sh 1}%
\special{pa 5220 1710}%
\special{pa 5153 1690}%
\special{pa 5167 1710}%
\special{pa 5153 1730}%
\special{pa 5220 1710}%
\special{fp}%
\put(45.2000,-17.0000){\makebox(0,0)[lb]{$p$}}%
\put(20.9400,-19.3000){\makebox(0,0)[lb]{$0$}}%
\put(20.9400,-7.1000){\makebox(0,0)[lb]{$1$}}%
\put(29.7400,-13.3000){\makebox(0,0)[lb]{$2$}}%
\put(45.3000,-13.4000){\makebox(0,0)[lb]{$4$}}%
\put(53.2000,-13.4000){\makebox(0,0)[lb]{$5$}}%
%
\special{pn 8}%
\special{pa 1530 1610}%
\special{pa 2860 1610}%
\special{dt 0.045}%
\special{pa 2860 1610}%
\special{pa 2859 1610}%
\special{dt 0.045}%
%
\special{pn 20}%
\special{pa 1840 1470}%
\special{pa 1809 1451}%
\special{pa 1777 1434}%
\special{pa 1747 1421}%
\special{pa 1717 1416}%
\special{pa 1688 1421}%
\special{pa 1661 1435}%
\special{pa 1637 1458}%
\special{pa 1617 1485}%
\special{pa 1602 1513}%
\special{pa 1591 1543}%
\special{pa 1584 1574}%
\special{pa 1581 1605}%
\special{pa 1579 1638}%
\special{pa 1579 1672}%
\special{pa 1580 1706}%
\special{pa 1581 1741}%
\special{pa 1585 1774}%
\special{pa 1594 1805}%
\special{pa 1610 1830}%
\special{pa 1637 1849}%
\special{pa 1669 1859}%
\special{pa 1702 1860}%
\special{pa 1732 1850}%
\special{pa 1759 1834}%
\special{pa 1786 1814}%
\special{pa 1790 1810}%
\special{sp}%
%
\special{pn 20}%
\special{pa 1770 1820}%
\special{pa 1860 1740}%
\special{fp}%
\special{sh 1}%
\special{pa 1860 1740}%
\special{pa 1797 1769}%
\special{pa 1820 1775}%
\special{pa 1823 1799}%
\special{pa 1860 1740}%
\special{fp}%
\put(15.8000,-13.8000){\makebox(0,0)[lb]{$\phi$}}%
%
\special{pn 20}%
\special{pa 3350 1590}%
\special{pa 3630 1590}%
\special{fp}%
%
\special{pn 20}%
\special{ar 3880 1600 232 232  0.0000000 6.2831853}%
\put(36.7000,-16.8000){\makebox(0,0)[lb]{$z-q$}}%
\put(37.6000,-13.3000){\makebox(0,0)[lb]{$3$}}%
\end{picture}%
\label{fig:D55}
\caption{Dynkin diagram of type $B_5^{(1)}$}
\end{figure}
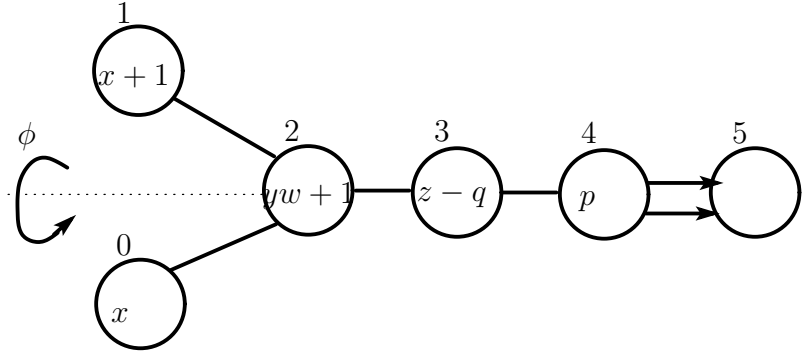

\begin{theorem}\label{3.4}
The system \eqref{10} admits extended affine Weyl group symmetry of type $B_5^{(1)}$ as the group of its B{\"a}cklund transformations {\rm (cf. \cite{NY2,S,T}), \rm} whose generators are explicitly given as follows{\rm : \rm}
\begin{align*}
        s_{0}: (*) &\rightarrow (x,y-\frac{\alpha_0}{x},z,w,q,p,t;-\alpha_0,\alpha_1,\alpha_2+\alpha_0,\alpha_3,\alpha_4,\alpha_5), \\
        s_{1}: (*) &\rightarrow (x,y-\frac{\alpha_1}{x+1},z,w,q,p,t;\alpha_0,-\alpha_1,\alpha_2+\alpha_1,\alpha_3,\alpha_4,\alpha_5), \\
        s_{2}: (*) &\rightarrow  (x+\frac{\alpha_2w}{yw+1},y,z+\frac{\alpha_2y}{yw+1},w,q,p,t;\alpha_0+\alpha_2,\alpha_1+\alpha_2,-\alpha_2,\alpha_3+\alpha_2,\alpha_4,\alpha_5), \\
        s_{3}: (*) &\rightarrow (x,y,z,w-\frac{\alpha_3}{z-q},q,p+\frac{\alpha_3}{z-q},t;\alpha_0,\alpha_1,\alpha_2+\alpha_3,-\alpha_3,\alpha_4+\alpha_3,\alpha_5), \\
        s_{4}: (*) &\rightarrow (x,y,z,w,q+\frac{\alpha_4}{p},p,t;\alpha_0,\alpha_1,\alpha_2,\alpha_3+\alpha_4,-\alpha_4,\alpha_5+\alpha_4), \\
        s_{5}: (*) &\rightarrow (x,y,z,w,q,p-\frac{2\alpha_5}{q}+\frac{t}{q^2},-t;\alpha_0,\alpha_1,\alpha_2,\alpha_3,\alpha_4+2\alpha_5,-\alpha_5), \\
        \phi: (*) &\rightarrow (-x-1,-y,-z,-w,-q,-p,-t;\alpha_1,\alpha_0,\alpha_2,\alpha_3,\alpha_4,\alpha_5).
\end{align*}
\end{theorem}

\begin{theorem}\label{3.5}
Let us consider a polynomial Hamiltonian system with Hamiltonian $H \in {\Bbb C}(t)[x,y,z,w,q,p]$. We assume that

$(A1)$ $deg(H)=4$ with respect to $x,y,z,w,q,p$.

$(A2)$ This system becomes again a polynomial Hamiltonian system in each coordinate $r_i \ (i=0,1,3,4,5)${\rm : \rm}
\begin{align*}
r_0&:x_0=-(xy-\alpha_0)y, \ y_0=1/y, \ z_0=z, \ w_0=w, \ q_0=q, \ p_0=p, \\
r_1&:x_1=-((x+1)y-\alpha_1)y, \ y_1=1/y, \ z_1=z, \ w_1=w, \ q_1=q, \ p_1=p, \\
r_3&:x_3=x, \ y_3=y, \ z_3=-((z-q)w-\alpha_3)w, \ w_3=1/w, \ q_3=q, \ p_3=p+w,\\
r_4&:x_4=x, \ y_4=y, \ z_4=z, \ w_4=w, \ q_4=1/q, \ p_4=-(pq+\alpha_4)q, \\
r_5&:x_5=x, \ y_5=y, \ z_5=z, \ w_5=w. \ q_5=q, \ p_5=p-\frac{2\alpha_5}{q}+\frac{t}{q^2}.
\end{align*}

$(A3)$ In addition to the assumption $(A2)$, the Hamiltonian system in the coordinate system $(x_0,y_0,z_0,w_0,q_0,p_0)$ becomes again a polynomial Hamiltonian system in the coordinate $r_2${\rm : \rm}
\begin{align*}
r_2:x_2=1/x_0, \ y_2=-((y_0+w_0)x_0+\alpha_2)x_0, \ z_2=z_0-x_0, \ w_2=w_0, \ q_2=q_0, \ p_2=p_0.
\end{align*}
Then such a system coincides with the system \eqref{10}.
\end{theorem}

Theorems \ref{3.4} and \ref{3.5} can be checked by a direct calculation, respectively.

\begin{theorem}\label{3.6}
For the system \eqref{4} of type $D_5^{(1)}$, we make the change of parameters and variables
\begin{gather}
\begin{gathered}\label{13}
A_0=\alpha_0, \quad A_1=\alpha_1, \quad A_2=\alpha_2, \quad A_3=\alpha_3, \quad A_4=\alpha_4, \quad A_5=\frac{\alpha_5-\alpha_4}{2},\\
\end{gathered}\\
\begin{gathered}\label{14}
X=x, \quad Y=y, \quad Z=z, \quad W=w, \quad Q=\frac{1}{q}, \quad P=-(pq+\alpha_4)q
\end{gathered}
\end{gather}
from $\alpha_0,\alpha_1, \dots ,\alpha_5,x,y,z,w,q,p$ to $A_0,A_1,\dots ,A_5,X,Y,Z,W,Q,P$. Then the system \eqref{4} can also be written in the new variables $X,Y,Z,W,Q,P$ and parameters $A_0,A_1,\dots ,A_5$ as a Hamiltonian system. This new system tends to the system \eqref{10} with the Hamiltonian \eqref{11}.
\end{theorem}

\begin{proof}
Notice that
$$
A_0+A_1+2A_2+2A_3+2A_4+2A_5=\alpha_0+\alpha_1+2\alpha_2+2\alpha_3+\alpha_4+\alpha_5=1
$$
and the change of variables from $(x,y,z,w,q,p)$ to $(X,Y,Z,W,Q,P)$ is symplectic. Choose $S_i \ (i=0,1,\dots ,5)$ and $\varphi$ as
$$
S_0:=s_0, \ S_1:=s_1, \ S_2:=s_2, \ S_3:=s_3, \ S_4:=s_4, \ S_5:=\pi_2, \ \phi:=\pi_1.
$$
Then the transformations $S_i$ are reflections of the parameters $A_0,A_1,\dots ,A_5$. The transformation group $\tilde{W}(B_5^{(1)})=<S_0,S_1,\dots ,S_5,\phi>$ coincides with the transformations given in Theorem \ref{3.4}.
\end{proof}

By using Theorems \ref{3.3} and \ref{3.6}, we show the relation between the system \eqref{5} and the system \eqref{10}.
\begin{theorem}\label{3.7}
For the system \eqref{5} of type $B_5^{(1)}$, we make the change of parameters and variables
\begin{gather}
\begin{gathered}\label{15}
A_0=\alpha_1, \quad A_1=2\alpha_0+\alpha_1, \quad A_2=\alpha_2, \quad A_3=\alpha_3, \quad A_4=\alpha_4, \quad A_5=\frac{\alpha_5-\alpha_4}{2},\\
\end{gathered}\\
\begin{gathered}\label{16}
X=-(xy-\alpha_1)y, \quad Y=\frac{1}{y}, \quad Z=z, \quad W=w, \quad Q=\frac{1}{q}, \quad P=-(pq+\alpha_4)q
\end{gathered}
\end{gather}
from $\alpha_0,\alpha_1, \dots ,\alpha_5,x,y,z,w,q,p$ to $A_0,A_1,\dots ,A_5,X,Y,Z,W,Q,P$. Then the system \eqref{5} can also be written in the new variables $X,Y,Z,W,Q,P$ and parameters $A_0,A_1,\dots ,A_5$ as a Hamiltonian system. This new system tends to the system \eqref{10} with the Hamiltonian \eqref{11}.
\end{theorem}

\section{The system of type $D_6^{(2)}$ }

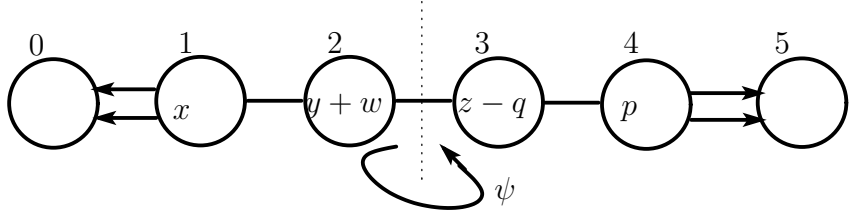
\begin{figure}[h]
\unitlength 0.1in
\begin{picture}(43.82,10.82)(16.10,-17.42)
%
\special{pn 20}%
\special{ar 2612 1182 232 232  0.0000000 6.2831853}%
%
\special{pn 20}%
\special{ar 1842 1194 232 232  0.0000000 6.2831853}%
%
\special{pn 20}%
\special{ar 4172 1189 232 232  0.0000000 6.2831853}%
\put(24.6600,-12.7600){\makebox(0,0)[lb]{$x$}}%
\put(39.5600,-12.7000){\makebox(0,0)[lb]{$z-q$}}%
%
\special{pn 20}%
\special{pa 4415 1190}%
\special{pa 4695 1190}%
\special{fp}%
%
\special{pn 20}%
\special{ar 4945 1200 232 232  0.0000000 6.2831853}%
%
\special{pn 20}%
\special{ar 5760 1190 232 232  0.0000000 6.2831853}%
%
\special{pn 20}%
\special{pa 5175 1140}%
\special{pa 5535 1140}%
\special{fp}%
\special{sh 1}%
\special{pa 5535 1140}%
\special{pa 5468 1120}%
\special{pa 5482 1140}%
\special{pa 5468 1160}%
\special{pa 5535 1140}%
\special{fp}%
\put(48.1500,-12.9000){\makebox(0,0)[lb]{$p$}}%
\put(17.1300,-9.4000){\makebox(0,0)[lb]{$0$}}%
\put(24.9400,-9.3000){\makebox(0,0)[lb]{$1$}}%
\put(40.4500,-9.3000){\makebox(0,0)[lb]{$3$}}%
\put(48.2500,-9.3000){\makebox(0,0)[lb]{$4$}}%
\put(56.1500,-9.3000){\makebox(0,0)[lb]{$5$}}%
%
\special{pn 20}%
\special{pa 2860 1180}%
\special{pa 3140 1180}%
\special{fp}%
%
\special{pn 20}%
\special{pa 2370 1120}%
\special{pa 2070 1120}%
\special{fp}%
\special{sh 1}%
\special{pa 2070 1120}%
\special{pa 2137 1140}%
\special{pa 2123 1120}%
\special{pa 2137 1100}%
\special{pa 2070 1120}%
\special{fp}%
%
\special{pn 20}%
\special{pa 2380 1270}%
\special{pa 2080 1270}%
\special{fp}%
\special{sh 1}%
\special{pa 2080 1270}%
\special{pa 2147 1290}%
\special{pa 2133 1270}%
\special{pa 2147 1250}%
\special{pa 2080 1270}%
\special{fp}%
%
\special{pn 20}%
\special{pa 5180 1280}%
\special{pa 5540 1280}%
\special{fp}%
\special{sh 1}%
\special{pa 5540 1280}%
\special{pa 5473 1260}%
\special{pa 5487 1280}%
\special{pa 5473 1300}%
\special{pa 5540 1280}%
\special{fp}%
%
\special{pn 20}%
\special{pa 3636 1420}%
\special{pa 3596 1427}%
\special{pa 3559 1435}%
\special{pa 3525 1444}%
\special{pa 3496 1455}%
\special{pa 3475 1469}%
\special{pa 3463 1487}%
\special{pa 3460 1508}%
\special{pa 3467 1532}%
\special{pa 3482 1558}%
\special{pa 3504 1585}%
\special{pa 3532 1612}%
\special{pa 3565 1638}%
\special{pa 3602 1662}%
\special{pa 3643 1683}%
\special{pa 3686 1701}%
\special{pa 3730 1716}%
\special{pa 3775 1727}%
\special{pa 3820 1735}%
\special{pa 3864 1740}%
\special{pa 3906 1742}%
\special{pa 3945 1741}%
\special{pa 3982 1737}%
\special{pa 4014 1730}%
\special{pa 4041 1721}%
\special{pa 4063 1709}%
\special{pa 4078 1694}%
\special{pa 4086 1677}%
\special{pa 4086 1658}%
\special{pa 4077 1637}%
\special{pa 4061 1614}%
\special{pa 4040 1589}%
\special{pa 4014 1564}%
\special{pa 3986 1538}%
\special{pa 3966 1520}%
\special{sp}%
%
\special{pn 20}%
\special{pa 3996 1540}%
\special{pa 3886 1430}%
\special{fp}%
\special{sh 1}%
\special{pa 3886 1430}%
\special{pa 3919 1491}%
\special{pa 3924 1468}%
\special{pa 3947 1463}%
\special{pa 3886 1430}%
\special{fp}%
\put(41.5000,-17.3000){\makebox(0,0)[lb]{$\psi$}}%
%
\special{pn 8}%
\special{pa 3770 660}%
\special{pa 3770 1590}%
\special{dt 0.045}%
\special{pa 3770 1590}%
\special{pa 3770 1589}%
\special{dt 0.045}%
%
\special{pn 20}%
\special{ar 3392 1182 232 232  0.0000000 6.2831853}%
\put(31.6000,-12.7000){\makebox(0,0)[lb]{$y+w$}}%
\put(32.7400,-9.3000){\makebox(0,0)[lb]{$2$}}%
%
\special{pn 20}%
\special{pa 3640 1180}%
\special{pa 3920 1180}%
\special{fp}%
\end{picture}%
\label{fig:D56}
\caption{Dynkin diagram of type $D_6^{(2)}$}
\end{figure}

In this section, we propose a 5-parameter family of coupled Painlev\'e III systems in dimension six with extended affine Weyl group symmetry of type $D_6^{(2)}$ given by
\begin{equation}\label{17}
  \left\{
  \begin{aligned}
   \frac{dx}{dt} &=\frac{2x^2y+2\alpha_0x}{t}-\frac{2(2xyz-\alpha_1z)}{t},\\
   \frac{dy}{dt} &=-\frac{2xy^2+2\alpha_0y-1}{t}+\frac{2y^2z}{t},\\
   \frac{dz}{dt} &=\frac{2z^2w+2(\alpha_0+\alpha_1+\alpha_2)z+t}{t}+\frac{2q(qp+\alpha_4)}{t},\\
   \frac{dw}{dt} &=-\frac{2zw^2+2(\alpha_0+\alpha_1+\alpha_2)w+1}{t}+\frac{2y(xy-\alpha_1)}{t},\\
   \frac{dq}{dt} &=\frac{2q^2p-(2\alpha_5-1)q+t}{t}+\frac{2wq^2}{t},\\
   \frac{dp}{dt} &=\frac{-2qp^2+(2\alpha_5-1)p}{t}-\frac{2(2wqp+\alpha_4w)}{t}
   \end{aligned}
  \right. 
\end{equation}
with the Hamiltonian
\begin{align}\label{18}
\begin{split}
H &=H_3(x,y,t;\alpha_0)+H_{III}^{D_7^{(1)}}(z,w,t,2(\alpha_0+\alpha_1+\alpha_2))\\
&+H_2(q,p,t;\alpha_5)+\frac{2(wq(qp+\alpha_4)-yz(xy-\alpha_1))}{t}\\
&=\frac{x^2y^2+2\alpha_0xy-x}{t}+\frac{z^2w^2+2(\alpha_0+\alpha_1+\alpha_2)zw+z+tw}{t}\\
&+\frac{q^2p^2-(2\alpha_5-1)qp+tp}{t}+\frac{2(wq(qp+\alpha_4)-yz(xy-\alpha_1))}{t}.
\end{split}
\end{align}
Here $x,y,z,w,q$ and $p$ denote unknown complex variables and $\alpha_0,\alpha_1, \dots ,\alpha_5$ are complex parameters satisfying the relation:
\begin{equation}\label{19}
\alpha_0+\alpha_1+\alpha_2+\alpha_3+\alpha_4+\alpha_5=\frac{1}{2}.
\end{equation}

\begin{theorem}\label{4.1}
The system \eqref{17} admits extended affine Weyl group symmetry of type $D_6^{(2)}$ as the group of its B{\"a}cklund transformations {\rm (cf. \cite{NY2,S,T}), \rm} whose generators are explicitly given as follows{\rm : \rm}
\begin{align*}
        s_{0}: (*) &\rightarrow (-x-\frac{2\alpha_0}{y}+\frac{1}{y^2},-y,-z,-w,-q,-p,-t;-\alpha_0,\alpha_1+2\alpha_0,\alpha_2,\alpha_3,\alpha_4,\alpha_5), \\
        s_{1}: (*) &\rightarrow (x,y-\frac{\alpha_1}{x},z,w,q,p,t;\alpha_0+\alpha_1,-\alpha_1,\alpha_2+\alpha_1,\alpha_3,\alpha_4,\alpha_5), \\
        s_{2}: (*) &\rightarrow  (x+\frac{\alpha_2}{y+w},y,z+\frac{\alpha_2}{y+w},w,q,p,t;\alpha_0,\alpha_1+\alpha_2,-\alpha_2,\alpha_3+\alpha_2,\alpha_4,\alpha_5), \\
        s_{3}: (*) &\rightarrow (x,y,z,w-\frac{\alpha_3}{z-q},q,p+\frac{\alpha_3}{z-q},t;\alpha_0,\alpha_1,\alpha_2+\alpha_3,-\alpha_3,\alpha_4+\alpha_3,\alpha_5),\\
        s_{4}: (*) &\rightarrow (x,y,z,w,q+\frac{\alpha_4}{p},p,t;\alpha_0,\alpha_1,\alpha_2,\alpha_3+\alpha_4,-\alpha_4,\alpha_5+\alpha_4),\\
        s_{5}: (*) &\rightarrow (x,y,z,w,q,p-\frac{2\alpha_5}{q}+\frac{t}{q^2},-t;\alpha_0,\alpha_1,\alpha_2,\alpha_3,\alpha_4+2\alpha_5,-\alpha_5), \\
        \psi: (*) &\rightarrow (-tp,\frac{q}{t},tw,-\frac{z}{t},-ty,\frac{x}{t},-t;\alpha_5,\alpha_4,\alpha_3,\alpha_2,\alpha_1,\alpha_0).
\end{align*}
\end{theorem}

\begin{theorem}\label{4.2}
Let us consider a polynomial Hamiltonian system with Hamiltonian $H \in {\Bbb C}(t)[x,y,z,w,q,p]$. We assume that

$(A1)$ $deg(H)=4$ with respect to $x,y,z,w,q,p$.

$(A2)$ This system becomes again a polynomial Hamiltonian system in each coordinate $r_i \ (i=0,1,\dots,5)${\rm : \rm}
\begin{align*}
r_0&:x_0=x+\frac{2\alpha_0}{y}-\frac{1}{y^2}, \ y_0=y, \ z_0=z, \ w_0=w, \ q_0=q, \ p_0=p, \\
r_1&:x_1=-(xy-\alpha_1)y, \ y_1=1/y, \ z_1=z, \ w_1=w, \ q_1=q, \ p_1=p, \\
r_2&:x_2=1/x, \ y_2=-((y+w)x+\alpha_2)x, \ z_2=z-x, \ w_2=w, \ q_2=q, \ p_2=p, \\
r_3&:x_3=x, \ y_3=y, \ z_3=-((z-q)w-\alpha_3)w, \ w_3=1/w, \ q_3=q, \ p_3=p+w,\\
r_4&:x_4=x, \ y_4=y, \ z_4=z, \ w_4=w, \ q_4=1/q, \ p_4=-(pq+\alpha_4)q, \\
r_5&:x_5=x, \ y_5=y, \ z_5=z, \ w_5=w. \ q_5=q, \ p_5=p-\frac{2\alpha_5}{q}+\frac{t}{q^2}.
\end{align*}
Then such a system coincides with the system \eqref{17}.
\end{theorem}

Theorems \ref{4.1} and \ref{4.2} can be checked by a direct calculation, respectively.

\begin{theorem}\label{4.3}
For the system \eqref{4} of type $D_5^{(1)}$, we make the change of parameters and variables
\begin{gather}
\begin{gathered}\label{20}
A_0=\frac{\alpha_1-\alpha_0}{2}, \quad A_1=\alpha_0, \quad A_2=\alpha_2, \quad A_3=\alpha_3, \quad A_4=\alpha_4, \quad A_5=\frac{\alpha_5-\alpha_4}{2},
\end{gathered}\\
\begin{gathered}\label{21}
X=-(xy-\alpha_0)y, \quad Y=\frac{1}{y}, \quad Z=z, \quad W=w, \quad Q=\frac{1}{q}, \quad P=-(pq+\alpha_4)q
\end{gathered}
\end{gather}
from $\alpha_0,\alpha_1, \dots ,\alpha_5,x,y,z,w,q,p$ to $A_0,A_1,\dots ,A_5,X,Y,Z,W,Q,P$. Then the system \eqref{4} can also be written in the new variables $X,Y,Z,W,Q,P$ and parameters $A_0,A_1,\dots ,A_5$ as a Hamiltonian system. This new system tends to the system \eqref{17} with the Hamiltonian \eqref{18}.
\end{theorem}

\begin{proof}
Notice that
$$
2(A_0+A_1+A_2+A_3+A_4+A_5)=\alpha_0+\alpha_1+2\alpha_2+2\alpha_3+\alpha_4+\alpha_5=1
$$
and the change of variables from $(x,y,z,w,q,p)$ to $(X,Y,Z,W,Q,P)$ is symplectic. Choose $S_i \ (i=0,1,\dots ,5)$ and $\varphi$ as
$$
S_0:=\pi_1, \ S_1:=s_0, \ S_2:=s_2, \ S_3:=s_3, \ S_4:=s_4, \ S_5:=\pi_2, \ \psi:=\pi_1\pi_2\pi_3.
$$
Then the transformations $S_i$ are reflections of the parameters $A_0,A_1,\dots ,A_5$. The transformation group $\tilde{W}(D_6^{(2)})=<S_0,S_1,\dots ,S_5,\psi>$ coincides with the transformations given in Theorem \ref{4.1}.
\end{proof}

\begin{flushleft}

\noindent
{\it \noindent Graduate School of Mathematical Sciences}

\noindent
{\it The University of Tokyo}

\noindent
{\it 3-8-1 Komaba Megro-ku 153-8914 Tokyo Japan}

\noindent
E-mail address: sasano@ms.u-tokyo.ac.jp

\end{flushleft}

\end{document}